\documentclass{amsart}
\usepackage{amssymb,amsmath,amsthm,amsfonts,epsfig,graphicx}
\usepackage[utf8]{inputenc}
\usepackage[T1]{fontenc}
\usepackage[english]{babel}
\usepackage{hyperref}

\newtheorem{teo}{Theorem}[section]

\newtheorem{lem}[teo]{Lemma}
\newtheorem{prop}[teo]{Proposition}
\theoremstyle{definition}
\newtheorem{rmk}[teo]{Remark}

\newcommand{\V}{{\mathcal V}}

\DeclareMathOperator{\fix}{fix}
\newcommand{\R}{\mathbb R}

\newcommand{\Z}{\mathbb Z}
\newcommand{\N}{\mathbb N}
\renewcommand{\epsilon}{\varepsilon}

\DeclareMathOperator{\diam}{diam}
\DeclareMathOperator{\dista}{dist} 
\DeclareMathOperator{\clos}{clos}
\title{Generic homeomorphisms with shadowing of one-dimensional continua}

\date{24 May 2019}
\author[A.~Artigue]{Alfonso Artigue}
\address{Departamento de Matemática y Estadística del Litoral, Universidad de la Rep\'ublica, Salto, Uruguay.}
\email{artigue@unorte.edu.uy}

\author[G.~Cousillas]{Gonzalo Cousillas}
\address{Instituto de Matemática y Estadística ``Rafael Laguardia'',
Facultad de Ingeniería, Universidad de la Rep\'ublica,
Montevideo, Uruguay.}
\email{gcousillas@fing.edu.uy}

\begin{document}
\begin{abstract}
	In this article we show that there are homeomorphisms of 
	plane continua whose conjugacy class is residual and have the shadowing property.
	\end{abstract}
\maketitle
%%%%%%%%%%%%%%%%%%%%%%%%%%%%%%%%%%%%%%%%%%
%\setcounter{section}{-1} %% Remove this when starting to work on the template.
%\section{How to Use this Template}
%The template details the sections that can be used in a manuscript. Note that the order and names of article sections may differ from the requirements of the journal (e.g., the positioning of the Materials and Methods section). Please check the instructions for authors page of the journal to verify the correct order and names. For any questions, please contact the editorial office of the journal or support@mdpi.com. For LaTeX related questions please contact latex@mdpi.com.
%The order of the section titles is: Introduction, Materials and Methods, Results, Discussion, Conclusions for these journals: aerospace,algorithms,antibodies,antioxidants,atmosphere,axioms,biomedicines,carbon,crystals,designs,diagnostics,environments,fermentation,fluids,forests,fractalfract,informatics,information,inventions,jfmk,jrfm,lubricants,neonatalscreening,neuroglia,particles,pharmaceutics,polymers,processes,technologies,viruses,vision

\section{Introduction}
Let $(X,\dista)$ be a compact metric space and 
denote by $\mathcal{H}(X)$ the space of homeomorphisms $f\colon X\to X$ with the $C^0$ distance 
\[
 \dista_{C^0}(f,g)=\sup\{\dista(f(x),g(x)),\dista(f^{-1}(x),g^{-1}(x)):x\in X\}.
\]
A property is said to be \emph{generic} if it holds on a residual subset of $\mathcal{H}(X)$. 
Recall that a set is $G_\delta$ if it is a countable intersection of open sets and 
it is \emph{residual} if it contains a dense $G_\delta$ subset. 
For instance, it is known that the shadowing property is generic for $X$ a compact manifold 
\cite[Theorem 1]{PiPl} or a Cantor set \cite[Theorem 4.3]{BeDa}. 
Recall that $f\in\mathcal{H}(X)$ has the \emph{shadowing property} if 
for all $\epsilon>0$ there is $\delta>0$ such that if $\{x_i\}_{i\in\Z}$ is a $\delta$-pseudo orbit then there is 
$y\in X$ such that $\dista(f^i(y),x_i)<\epsilon$ for all $i\in\Z$. 
We say that $\{x_i\}_{i\in\Z}$ is a $\delta$-\emph{pseudo orbit} if $\dista(f(x_i),x_{i+1})<\delta$ for all $i\in\Z$.

A remarkable result, proved in \cite{AGW,KeRo}, states that 
if $X$ is a Cantor set then there is a homeomorphism of $X$ whose conjugacy class is a dense $G_\delta$ subset of $\mathcal{H}(X)$. 
That is, a generic homeomorphism of a Cantor set is conjugate to this special homeomorphism. 
We say that $f,g\in\mathcal{H}(X)$ are \emph{conjugate} if there is $h\in\mathcal{H}(X)$ such that 
$f\circ h=h\circ g$ and the \emph{conjugacy class} of $f$ is the set of all the homeomorphisms conjugate to $f$. 
This result gives rise to a natural question: 
besides the Cantor set, 
\begin{center}
\emph{which compact metric spaces have a $G_\delta$ dense conjugacy class? } 
\end{center}
On a space with a $G_\delta$ dense conjugacy class, the study of generic properties (invariant under conjugacy, as the shadowing property) is 
reduced to determine whether a representative of this class has the property or not.
% \azul{It is straightforward to see that the shadowing property is conjugacy invariant when $X$ is a compact metric space.
% Then if there exists a $G_\delta$ conjugacy class, and a representative element has the shadowing property we can conclude that shadowing is generic on $X$, then another natural question related to the former one is
% \begin{center}
%  \emph{ do their elements have shadowing?}
% \end{center}
% }

In Theorem \ref{teoGenShEnY} 
we show that there are one-dimensional plane continua with a $G_\delta$ dense conjugacy class 
whose members have the shadowing property.
The proof of this result is based on Theorem \ref{teoElGenerico} where we show that 
for a compact interval $I$ there is a $G_\delta$ conjugacy class in $\mathcal{H}(I)$ 
which is dense in the subset of orientation preserving homeomorphisms of $I$. 
In addition, the proof of Theorem \ref{teoGenShEnY} depends on Propositions \ref{propElGenericoUniDim} and \ref{propCondSufSh}
where we give sufficient conditions for the existence of a residual conjugacy class and for a homeomorphism to have 
the shadowing property, respectively. 
The following open question has an affirmative answer in the examples known by the authors:
\begin{center}
 \emph{if a homeomorphism has a $G_\delta$ dense conjugacy class, does it have the shadowing property?}
\end{center}

%%%%%%%%%%%%%%%%%%%%%%%%%%%%%%%%%%%%%%%%%%

\section{Generic dynamics on a closed segment}
Let $I=[0,1]$ and define $\mathcal{H}^+(I)=\{f\in \mathcal{H}(I):f\text{ preserves orientation}\}$.
In this section we will show the next result. 
\begin{teo}
\label{teoElGenerico}
There is $f_*\in\mathcal{H}^+(I)$ whose conjugacy class is 
a $G_\delta$ dense subset of $\mathcal{H}^+(I)$.
\end{teo}

\begin{rmk}
The generic dynamics of circle homeomorphisms is studied in detail in \cite[Theorem 9.1]{AHK}.
The proof of Theorem \ref{teoElGenerico} follows the same ideas. 
As we could not find this result in the literature we include the details.
\end{rmk}

To prove Theorem \ref{teoElGenerico} we start by defining the homeomorphism $f_*$. 
For this purpose we introduce some definitions.
For $f\in\mathcal{H}^+(I)$ let $\fix(f)=\{x\in X:f(x)=x\}$. 
A connected component of $I\setminus \fix(f)$ will be called \emph{wandering interval}. 
Following \cite{Pi}, we say that a wandering interval $(a,b)$ is an $r$-\emph{interval}  if $\lim_{n\to+\infty}f^n(x)=b$ for all $x\in (a,b)$. 
Analogously, it is an $l$-\emph{interval} if $\lim_{n\to+\infty}f^n(x)=a$ for all $x\in (a,b)$. 
For each interval $[a,b]$ fix a homeomorphism
$f_r^{[a,b]}:[a,b]\to[a,b]$ such that $(a,b)$ is an $r$-interval. Analogously, we consider 
$f_l^{[a,b]}$ with $(a,b)$ an $l$-interval.

For $n\geq 0$ and $0\leq k< 3^{n}$ define the closed interval
\[
J(n,k)=
\left[\tfrac{3k+1}{3^{n+1}},\tfrac{3k+2}{3^{n+1}}\right]. 
\]
For $x$ in the ternary Cantor set define $f_*(x)=x$. 
In other case, there is a minimum integer $n_x\geq 0$ such that $x\in J(n_x,k)$ for some 
$0\leq k< 3^{n_x}$ and define 
\[
 f_*(x)=\left\{
 \begin{array}{ll}
  f_l^{J(n_x,k)}(x) & \text{ if }n_x\text{ is odd},\\
  f_r^{J(n_x,k)}(x) & \text{ if }n_x\text{ is even}.
 \end{array}
 \right.
\]
For example, $(\tfrac{1}{3},\tfrac{2}{3})$ is an $r$-interval, while 
$(\tfrac{1}{3^2},\tfrac{2}{3^2})$ and $(\tfrac{7}{3^2},\tfrac{8}{3^2})$ are $l$-intervals.
See Figure \ref{figHomeoG}.
\begin{figure}[h]
\centering
 \includegraphics[scale=1.2]{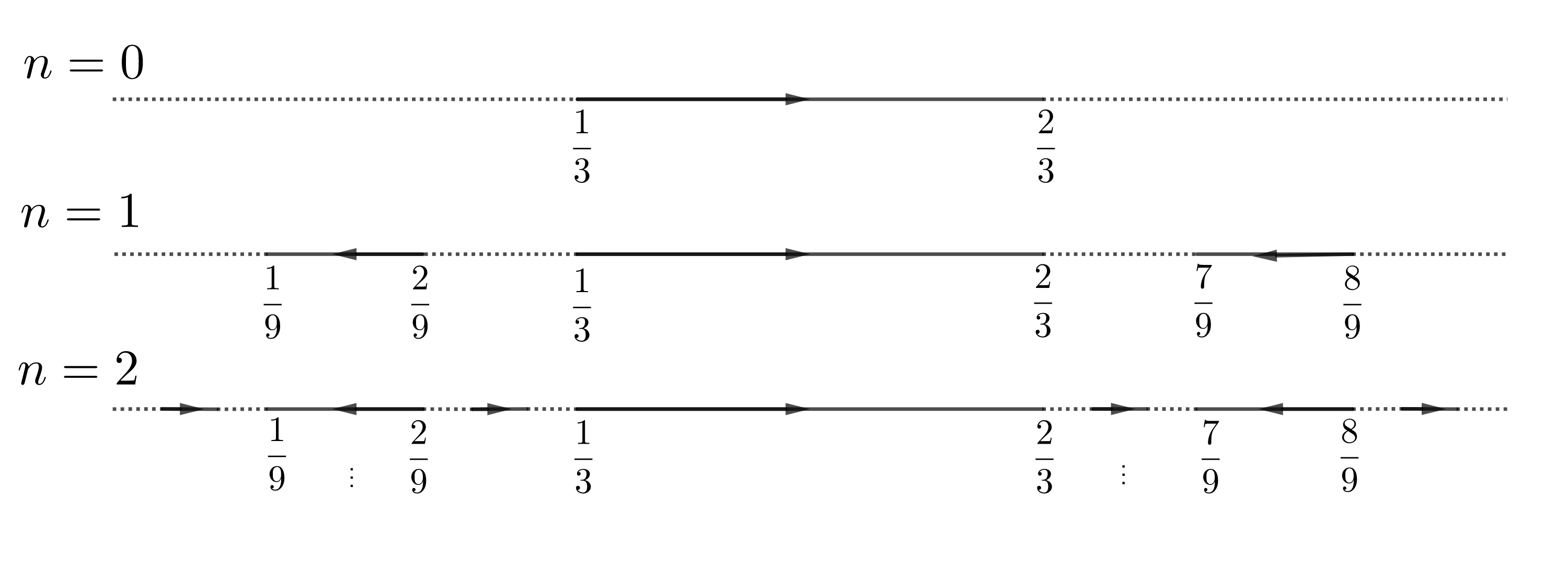}
 \caption{A sketch of the phase diagram of $f_*$.}
 \label{figHomeoG}
\end{figure}
%\rojo{Cambiar en la figura: $n=0,1,2$}

\begin{rmk}
\label{rmkShFest}
From \cite[Theorem 8]{PeVa} we know that $f_*$, and every homeomorphism conjugate to $f_*$, has the shadowing property. 
\end{rmk}

The next result gives a useful characterization of the conjugacy class of $f_*$.
Given $\epsilon>0$ we say that $g\in\mathcal{H}^+(I)$ satisfies the property $P_\epsilon$ if  there are intervals $J_i=(a_i,b_i)$, $i=1,\dots,n$, such that:
 \begin{enumerate}
  \item $0<a_1<b_1<a_2<b_2<a_3<\dots<b_n<1$,
  \item $J_i$ is an $r$-interval for $i$ odd and an $l$-interval for $i$ even,
  \item $\max\{a_1,1-b_n\}<\epsilon$ and $\max\{a_{i+1}-b_i:1\leq i<n\}<\epsilon$.
 \end{enumerate} 

\begin{prop}
\label{propCharClassf*}
A homeomorphism $g\in\mathcal{H}^+(I)$ is conjugate to $f_*$ if and only if 
it satisfies $P_\epsilon$ for all $\epsilon>0$.
% define $\G_\epsilon^+\subset \mathcal{H}^+(I)$ to be the set of  $f\in \mathcal{H}^+(I)$ satisfying
\end{prop}

\begin{proof}
The direct part of the proof is clear from the construction of $f_*$. 

To prove the converse suppose that $g$ satisfies $P_\epsilon$ for all $\epsilon>0$.
 From the condition (3) we see that $\fix(g)$ is totally disconnected. 
 Suppose that $p\in I$ is an isolated fixed point. 
 If $p=0$ then there is a wandering interval $(0,x)$. Taking $\epsilon\in(0,x)$ we have a contradiction with (3), because $a_1<\epsilon$.  
 Analogously we show that $p$ cannot be 1. 
 If $p\in (0,1)$ then $p$ is in the boundary of two wandering intervals. Taking $\epsilon$ smaller than the length of these intervals 
 we contradict (1) and (3).
 Thus, $\fix(g)$ has no isolated point and is a Cantor set. 
 Condition (2) (applied for a suitable $\epsilon$ small) implies that between two wandering intervals there is 
 an $r$-interval and an $l$-interval. 
% it satisfies the following conditions: 
% \begin{enumerate}
%  \item $\fix(f)$ is a Cantor set and \label{cond_1_G}
%  \item between two wandering intervals there is 
%  an $r$-interval and an $l$-interval.
% \end{enumerate}

Let $\mathcal{R}$ and $\mathcal{L}$ be the families of $r$-intervals  and $l$-intervals of $g$ respectively.
We define an order in $\mathcal{R}\cup\mathcal{L}$ in the following way: 
$I_\alpha < I_\beta$ if $x< y$ for all $x\in I_\alpha$, $y\in I_\beta$.
We will make the conjugacy by induction.
For the first step  name $I_{1/2}\in\mathcal{R}$ which satisfies $\diam(I_{1/2})\geq \diam(I)$ 
for every $I\in \mathcal{R}$.
In case that there exists more than one interval which verifies this condition we choose any of them.
Let $J_{c}$ be a wandering interval of $f_*$ 
such that $c$ is the midpoint of $J_c$. 
By construction, $J_{1/2}$ is an $r$-interval of $f_*$, thus we can 
consider a conjugacy $h_{1/2}\colon I_{1/2}\to J_{1/2}$ of $g$ and $f_*$ restricted to these intervals.
Notice that $1/6$ and $5/6$ are the midpoints of $(1/9,2/9)$ and $(7/9,8/9)$ respectively.
Take $I_{1/6}\in\mathcal{L}$ satisfying $I_{1/6}<I_{1/2}$ and $\diam(I_{1/6})\geq \diam(I)$ 
for every $I\in\mathcal{L}$ such that $I<I_{1/2}$. 
Also, take $I_{5/6}\in\mathcal{L}$ satisfying $I_{1/2}<I_{5/6}$ 
and $\diam(I_{5/6})\geq \diam(I)$ for every $I\in\mathcal{L}$ such that $I>I_{1/2}$.
Then consider $h_{1/6}\colon I_{1/6}\to J_{1/6}$ to be a conjugacy from $g$ to $f_*$ restricted to the corresponding intervals.
Similarly, define $h_{5/6}$.
%For $n=2$, we have to define $2^2$ homeomorphisms.
%Name $I_{1/8}\in\mathcal{R}$ satisfying $I_{1/8}< I_{1/4}$, $\diam(I_{1/8})\geq \diam(I_n)$ for every $I_n\in\mathcal{R}$, $I_n<I_{1/4}$; $I_{3/8}\in\mathcal{R}$ satisfying $I_{1/4}<I_{3/8}<I_{1/2}$, $\diam(I_{3/8})\geq \diam(I_n)$ for every $I_n\in\mathcal{R}$, $I_{1/4}<I_{n}<I_{1/2}$; $I_{5/8}\in\mathcal{R}$ satisfying $I_{1/2}<I_{5/8}<I_{3/4}$, $\diam(I_{5/8})\geq \diam(I_n)$ for every $I_n\in\mathcal{R}$, $I_{1/2}<I_{n}<I_{3/4}$; $I_{7/8}\in\mathcal{R}$ satisfying $I_{3/4}<I_{7/8}$, $\diam(I_{7/8})\geq \diam(I_n)$ for every $I_n\in\mathcal{R}$, $I_{3/4}<I_{n}$.
%Then we define $h_{k/4}:I_{k/8}\to J_{k/8}$ in the standar way for $k=1,3,5,7$. We then go on for each $n\in \N$.
Then we go on defining $2^{k-1}$ homeomorphisms on each step. 
If $k-1$ is even, we choose $r$-intervals, otherwise we choose $l$-intervals.
Notice that  since  in each step we choose the largest interval of the $r$ or $l$-intervals of $g$, 
every wandering interval of $g$ is eventually chosen. 
In this way, the conjugacies $h_{j/k}$ give rise to a conjugacy $h$ of $g$ and $f_*$ in the whole interval $[0,1]$ and the proof ends.
% Then $h$ is defined in an open and dense subset of $[0,1]$, so it extends to the unit interval.
\end{proof}

\begin{proof}[Proof of Theorem \ref{teoElGenerico}]
Given $n\geq 1$, let $\mathcal{U}_n$ be the set of increasing homeomorphisms of $I$ satisfying $P_{1/n}$.
Notice that $P_\epsilon$ implies $P_{\epsilon'}$ for all $\epsilon'>\epsilon>0$.
Thus, from Proposition \ref{propCharClassf*} we have that 
the conjugacy class of $f_*$ is the countable intersection
$\bigcap_{n\geq 1} \mathcal{U}_n$.
To finish the proof, applying Baire's Theorem, 
we will show that each $\mathcal{U}_n$ is open and dense in $\mathcal{H}^+(I)$.
 
 To prove that $\mathcal{U}_n$ is open consider $f\in\mathcal{U}_n$. It is clear that there is $\delta>0$ such that 
 $f\in\mathcal{U}_{n-4\delta}$. 
 Consider the intervals $(a_i,b_i)$ from the definition of property $P_\epsilon$, for $\epsilon=1/n$. 
For each odd $i=1,\dots,n$, 
take $x_i\in (a_i,a_i+\delta)$ and
for $i$ even take $y_i\in (b_i-\delta,b_i)$. 
Consider $m\in\N$ large such that $f^m(x_i)\in(b_i-\delta,b_i)$ and 
$f^m(y_i)\in (a_i,a_i+\delta)$ for all $i$.
Take a neighborhood $\V$ of $f$ such that
$\dista_{C^0}(f^m,g^m)<\delta$ for all $g\in\V$ and
$g^m(x_i)>x_i$, $g^m(y_i)<y_i$ for all $i$.
This implies that $(x_i,g^m(x_i))$ is contained in an $r$-interval for $g$ and 
$(g^m(y_i),y_i)$ is contained in an $l$-interval for $g$.
For all $g\in \V$ and $i$ odd we have 
\[
 \begin{array}{lll}
|g^m(x_i)-g^m(y_{i+1}))|&\leq &|g^m(x_i)-f^m(x_i)|+|f^m(x_i)-f^m(y_{i+1})|\\
			&&+|f^m(y_{i+1})-g^m(y_{i+1})|\\
			&<&\delta+ |f^m(x_i)-b_i|+|b_i-a_{i+1}|+|a_{i+1}-f^m(y_{i+1})|+ \delta  \\
			&<&2\delta+(1/n-4\delta)+2\delta =1/n. 
 \end{array}
\]
Arguing analogously for $i$ even, we conclude that $g\in\mathcal{U}_n$ and $\mathcal{U}_n$ is open. 

To prove that $\mathcal{U}_n$ is dense in $\mathcal{H}^+(I)$ the following remark is sufficient. 
Given $f\in \mathcal{H}^+(I)$, $p\in\fix(f)\cap (0,1)$ and $\delta>0$ small, we can define $g\in \mathcal{H}^+(I)$ close to $f$ such that: 
\begin{itemize}
 \item $f|_{[0,p]}$ and $g|_{[0,p-\delta]}$ are conjugate,
 \item $f|_{[p,1]}$ and $g|_{[p+\delta,1]}$ are conjugate and 
 \item $g$ has an $r$ or $l$-interval at $[p-\delta,p+\delta]$. 
\end{itemize}
That is, a fixed point can be \emph{exploded} into a small wandering interval with an arbitrarily small perturbation. 
Performing finitely many of such explosions the density of $\mathcal{U}_n$ is obtained.
\end{proof}

% \begin{rmk}
% (The orientation reversing case).
% Let $$\mathcal{H}^-(I)=\{f\in\mathcal{H}(I): f \mbox{ reverses orientation}\}.$$
% %Take $f\in \mathcal{H}^-(I)$, a connected component of $I\setminus Per(f)$ will be called {\it wandering interval}.
% Note that if $f\in \mathcal{H}^-(I)$, then $f^2\in \mathcal{H}^+(I)$ and if $(a,b)$ is a wandering interval of $f$, then $(a,b)$ is a wandering interval of $f^2$.
% Define $\mathcal{G}^-$ by $f\in \mathcal{H}^-$ such that:
% \begin{enumerate}
%  \item $Per(f)$ is a Cantor set.
%  \item between two wandering intervals of $f^2$ there is an $r$-interval and an $l$-interval.
%  \end{enumerate}
% It holds that:
%  There is a $G_\delta$ conjugacy class in $\mathcal{H}(I)$ which is dense in $\mathcal{H}^-$. 
% \end{rmk}

\section{Genericity on a plane one-dimensional continuum}
\label{secElGenericoUniDim}
In this section we show that there are some particular 
one-dimensional plane continua with a $G_\delta$ dense conjugacy class whose members have the shadowing property. 
We start with a sufficient condition for the existence of a $G_\delta$ dense conjugacy class.
An open subset $U\subset X$ is a \emph{free arc} if it is homeomorphic to $\R$. 

\begin{prop}
\label{propElGenericoUniDim}
If $X$ is a compact metric space such that 
\begin{enumerate}
 \item $X=\cup_{n\geq 1}a_n$, where each $a_n$ is a compact arc with extreme points $p_n,q_n\in X$ for all $n\geq 1$,
 \item $a_n\setminus \{p_n,q_n\}$ is a free arc for all $n\geq 1$ and
 \item for all $f\in\mathcal{H}(X)$ it holds that $f(a_n)=a_n$ and $p_n,q_n\in\fix(f)$ for all $n\geq 1$
\end{enumerate}
then $\mathcal{H}(X)$ has a $G_\delta$ dense conjugacy class.
\end{prop}

\begin{proof}
%  In this case, for the continuum we are considering, every singular point is fixed by any $f\in\mathcal{H}(X)$. 
%  Also, the maximal free arcs (i.e., those limited by singular points) are invariant by any $f\in\mathcal{H}(X)$. 
%  Moreover, the restriction of any $f\in\mathcal{H}(X)$ to a maximal free arc preserves orientation. 
%  
%  Let $a_1,a_2,\dots$ be the closed arcs shown in Figure \ref{figConjuntoX}. 
 For each $n\geq 1$ let $X_n=\clos(X\setminus a_n)$ and define 
 \[
  \mathcal{H}_n=\{f\in\mathcal{H}(X_n): p_n,q_n\in\fix(f)\}
 \]
and the map $\varphi_n\colon \mathcal{H}(X)\to \mathcal{H}^+(a_n)\times \mathcal{H}_n$ as 
$\varphi_n(f)=f|_{a_n}\times f|_{X_n}$.
In $\mathcal{H}^+(a_n)\times \mathcal{H}_n$ we consider the product topology.
It is clear that $\varphi_n$ is a homeomorphism for each $n\geq 1$.
Let $\mathcal{R}_n$ be the $G_\delta$ dense conjugacy class of $\mathcal{H}^+(a_n)$ 
given by Theorem \ref{teoElGenerico}
and define 
$\mathcal{S}_n=\mathcal{R}_n\times \mathcal{H}_n$. 
Thus, $\cap_{n\geq 1} \varphi_n^{-1}(\mathcal{S}_n)$ is a $G_\delta$ dense conjugacy class in $\mathcal{H}(X)$. 
\end{proof}

\begin{rmk}\label{rmkRepGen}
Notice that a representative $g_*$ of the $G_\delta$ dense conjugacy of Proposition \ref{propElGenericoUniDim} 
is obtained by considering a conjugate of 
$f_*$ on each arc $a_n$ of $X$. 
\end{rmk}

Now we will prove a sufficient condition for a homeomorphism to have the shadowing property.
For this purpose we need some definitions and a lemma.
Suppose that $(X,\dista)$ is a compact metric space and take $f\in\mathcal{H}(X)$. 
A compact $f$-invariant subset $A\subset X$ is a \emph{quasi-attractor} if for every 
open neighborhood $U$ of $A$ 
there is an open subset $V\subset U$ such that $A\subset V$ and 
$\clos(f(V))\subset V$. 
If in addition $f\colon A\to A$ has the shadowing property we say that $A$ is a \emph{quasi-attractor with shadowing}.

\begin{lem}
\label{lemShQA}
 If $A\subset X$ is a quasi-attractor with shadowing then for all $\epsilon>0$ 
 there is $\delta>0$ such that if $\{x_n\}_{n\geq 0}$ is a $\delta$-pseudo-orbit
 with $x_0\in B_\delta(A)$ then there is $y\in A$ that $\epsilon$-shadows $\{x_n\}_{n\geq 0}$. 
\end{lem}

\begin{proof}
Given $\epsilon>0$ 
take $\delta_1>0$ such that every $\delta_1$-pseudo-orbit in $A$ is $\epsilon/2$-shadowed by a point in $A$.
Consider $0<\alpha<\min\{\epsilon/2,\delta_1/3\}$ 
such that $\dista(a,b)<\alpha$ implies $\dista(f(a),f(b))<\delta_1/3$. 
Since $A$ is a quasi-attractor, for $U=B_\alpha(A)$ there exists an open set $V$ such that
$A\subset V\subset U$ and $\clos(f(V))\subset V$. 
Take $\delta\in(0,\delta_1/3)$ such that $B_\delta(\clos(f(V)))\subset V$.

Suppose that $\{x_n\}_{n\geq 0}$ is a $\delta$-pseudo-orbit
 with $x_0\in B_\delta(A)$.
Since $f(x_0)\in f(V)$ we have that 
$x_1\in B_\delta(f(V))$ and $x_1\in V$. 
In this way we prove that
$x_n\in V$ for all $n\geq 0$. 
For each $n\geq 0$ take $y_n\in A$ such that $\dista(y_n,x_n)<\alpha$. 
We have that
\[
\begin{array}{ll}
 \dista(f(y_n),y_{n+1})&\leq 
 \dista(f(y_n),f(x_n))+\dista(f(x_n),x_{n+1})+\dista(x_{n+1},y_{n+1})\\
 &\leq\delta_1/3+\delta+\alpha< 3\delta_1/3=\delta_1.
\end{array}
\]
This proves that $\{y_{n}\}_{n\geq 0}$ is a $\delta_1$-pseudo-orbit contained in $A$. 
There exists $z\in A$ that $\epsilon/2$-shadows $\{y_n\}_{n\geq 0}$. 
Thus
\[
 \dista(f^n(z),x_n)\leq\dista(f^n(z),y_n)+\dista(y_n,x_n)<\epsilon/2+\alpha\leq\epsilon.
\]
And the proof ends.
\end{proof}

\begin{prop}
\label{propCondSufSh}
 If every point of $X$ belongs to a quasi-attractor with shadowing then $f$ has shadowing.
\end{prop}

\begin{proof}
Suppose that $\epsilon>0$ is given. 
For each $x\in X$ let $A_x\subset X$ be a quasi-attractor with shadowing containing $x$. 
Let $\delta_x>0$ be given by Lemma \ref{lemShQA} such that for every $\delta_x$-pseudo-orbit $\{x_n\}_{n\geq 0}$ with 
$x_0\in B_{\delta_x}(A_x)$ there is a point in $A_x$ that $\epsilon$-shadows $\{x_n\}_{n\geq 0}$.
As $X$ is compact there is a finite sequence $x_1,\dots,x_k\in X$ such that 
$\cup_{i=1}^k B_{\delta_i}(A_i)=X$ where $A_i=A_{x_i}$ and $\delta_i=\delta_{x_i}$.
If we take $\delta=\min\{\delta_1,\dots,\delta_k\}$ we have that for every $\delta$-pseudo-orbit $\{x_n\}_{n\geq 0}$ in $X$
there is $j$ such that $x_0\in B_{\delta_j}(A_j)$. Then, there is a point in $A_j$ that $\epsilon$-shadows $\{x_n\}_{n\geq 0}$ 
and the proof ends.
\end{proof}

Let $Y\subset\R^2$ be the union of: 
\begin{itemize}
 \item the circle arc $x^2+y^2=1$, $y\leq 0$,
 \item the horizontal segment $[-1,1]\times\{0\}$ and 
 \item the vertical segments $\{-1+\frac{2}{n}\}\times [0,1/n]$, for $n\geq 1$.
\end{itemize}
See Figure \ref{figConjuntoX}. 
\begin{figure}[ht]
\centering
 \includegraphics[scale=0.3]{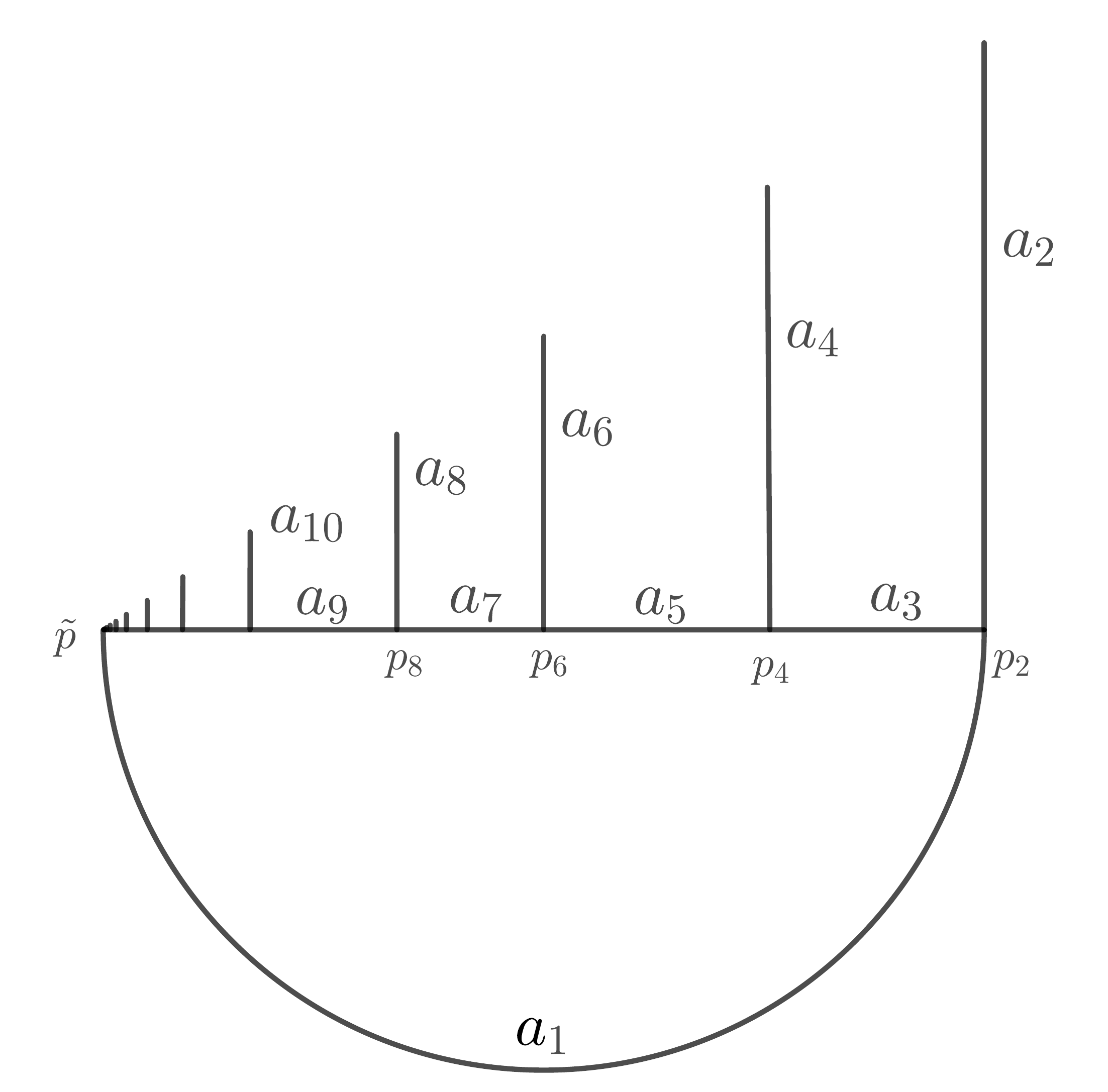}
 \caption{The continuum $Y$ can be decomposed as a union of arcs as in Proposition \ref{propElGenericoUniDim}.}
 \label{figConjuntoX}
\end{figure}

\begin{teo}
\label{teoGenShEnY}
For the continuum $Y$ 
% shown in Figure \ref{figConjuntoX} 
there is a $G_\delta$ conjugacy class which is dense in $\mathcal{H}(Y)$ and whose members have the shadowing property. 
In particular, the shadowing property is generic in $\mathcal{H}(Y)$.
\end{teo}

\begin{proof}
The continuum $Y$ satisfies the hypothesis of Proposition \ref{propElGenericoUniDim}. 
Indeed, the conditions $(1)$ and $(2)$ are direct from the construction of $Y$. 
Consider the points $p_n,\tilde p$ indicated in Figure \ref{figConjuntoX}. 
It is clear that $\tilde p \in\fix(f)$ for all $f\in\mathcal{H}(Y)$. 
This implies that $a_1$ is invariant and $p_2\in\fix(f)$. 
In turn, this implies that $a_2$ is invariant under each $f\in\mathcal{H}(Y)$. 
In this way it is shown that condition 
$(3)$ of Proposition \ref{propElGenericoUniDim} holds. 
Therefore, $\mathcal{H}(Y)$ contains a $G_\delta$ dense conjugacy class.

As explained in Remark \ref{rmkRepGen}, a representative $g_*\in\mathcal{H}(Y)$ of this conjugacy class is obtained by taking a conjugate 
of $f_*$ on each arc $a_n$. 
It only remains to prove that $g_*$ has the shadowing property. 
By Remark \ref{rmkShFest} we know that $g_*\colon a_n\to a_n$ has the shadowing property. 
By construction, each $a_n$ is a quasi-attractor for $g_*$. Since the arcs $a_n$ cover $Y$ we can apply 
Proposition \ref{propCondSufSh} to conclude that $g_*$ has the shadowing property.
\end{proof}

\end{document}